\documentclass[reqno,oneside,12pt]{amsart}

\usepackage[T1]{fontenc}
\usepackage{times,mathptm}
\usepackage{amscd,color, enumerate}
\usepackage[svgnames]{xcolor}

\usepackage{amssymb}
\usepackage{mathtools}
\usepackage{stackrel}

\usepackage[colorlinks=true, pdfstartview=FitV, linkcolor=blue, citecolor=blue, urlcolor=blue]{hyperref}

\title[On a theorem of Ramanujam]{On Ramanujam's theorem about finite dimensional groups of automorphisms}

\author[Cantat et al]{Serge Cantat, Hanspeter Kraft, \\Andriy Regeta, Immanuel van Santen}

\address{Univ. Rennes, CNRS, IRMAR - UMR 6625,\newline\indent
35000 Rennes, Frances}
\email{serge.cantat@univ-rennes.fr}
\address{Departement Mathematik und Informatik,\newline\indent
Universit\"at Basel,  Spiegelgasse~1, CH-4051 Basel}
\email{hanspeter.kraft@unibas.ch}
\address{Dipartimento di Matematica ``Tullio Levi-Civita'',\newline\indent
Universit\`a di Padova, Via Trieste 63, I-35121 Padova}
\email{andriyregeta@gmail.com}
\address{Departement Mathematik und Statistik,\newline\indent
Universit\"at Bern, Sidlerstrasse 5, CH-3012 Bern}
\email{immanuel.van.santen@gmail.com}
\date{2024-2026}

\thanks{The work of S.C. is supported by  the European Research Council (ERC GOAT 101053021). The work of A.R. was partially supported by  DFG, project number 509752046.
The work of I.v.S. was supported by the SNSF grant 10000940.}

\theoremstyle{plain}
\newtheorem{thm}{Theorem}[section]

\newtheorem{ramthm}{Ramanujam's theorem\!\!}

\newtheorem{prop}[thm]{Proposition} 
\newtheorem{lem}[thm]{Lemma}
\newtheorem*{lem*}{Lemma}

\newtheorem{cor}[thm]{Corollary}

\theoremstyle{definition}

\newtheorem{exa}[thm]{Example}
\newtheorem{rem}[thm]{Remark}

\newcommand{\mm}{{\mathfrak m}}

\newcommand{\bbA}{{\mathbb A}}

\newcommand{\OOO}{\mathcal{O}}
\newcommand{\GGG}{\mathcal{G}}
\newcommand{\GGGr}[1]{\GGG\vert_{#1}}

\newcommand{\SSS}{\mathcal{S}}

\newcommand{\simto}{\xrightarrow{\sim}}
\newcommand{\be}{
\begin{enumerate}}
\newcommand{\ee}{
\end{enumerate}}

\renewcommand{\phi}{\varphi}

\DeclareMathOperator{\id}{id}

\DeclareMathOperator{\Aut}{Aut}

\DeclareMathOperator{\Spec}{Spec}

\DeclareMathOperator{\Ga}{\mathbf{G}_{a}}

\DeclareMathOperator{\Char}{char}

\DeclareMathOperator{\Lie}{Lie}

\DeclareMathOperator{\Ve}{Vec}

\DeclareMathOperator{\mdo}{{{\mathrm{mdo}}}}

\renewcommand{\subset}{\subseteq}
\newcommand{\kk}{\Bbbk}
\newcommand{\KK}{\mathbb K}

\newcommand{\Xm}{X_{\text{\rm{max}}}}

\renewcommand{\aa}{\mathfrak{a}}

\newcommand{\ps}{\par\smallskip}

\newcommand{\set}[2]{\left\{\,#1 \ | \ #2\,\right\}}

\addtocounter{section}{0} 
\numberwithin{equation}{subsection} 


\begin{document}

\begin{abstract}
Ramanujam's theorem states that any connected finite-di\-men\-sional subgroup of the automorphism group $\Aut(X)$ of an irreducible variety~$X$  is an algebraic group, 
in a natural way. In this note, we discuss the notion of dimension and extend  
Ramanujam's theorem to arbitrary (not necessarily irreducible) varieties.
\end{abstract}
 
\maketitle

\setcounter{tocdepth}{1}
\tableofcontents

\section{Introduction}
%
Let $X$ be an arbitrary (possibly reducible) 
variety over an algebraically closed field $\kk$. 
Denote by $\Aut(X)$ its group of algebraic automorphisms. 

\subsection{Morphisms, connectedness, and dimension} 

The following notion goes back to Ramanujam's article~\cite{Ra1964A-note-on-automorp}.

For any variety $A$, a \emph{morphism} $\rho \colon A \to \Aut(X)$ is a map such that the induced map $A \times X \to X$, $(a, x) \mapsto \rho(a)(x)$ is a morphism of varieties.  This is equivalent to  $(a, x) \mapsto (a, \rho(a)(x))$ being an automorphism of $A \times X$ (see Lemma~\ref{Lem.inverse-family}  in the Appendix). The image $V$ of a morphism $\rho \colon A \to \Aut(X)$ is a {\emph{constructible}} subset of $\Aut(X)$. If we can choose $A$ to be irreducible,
then $V$ is an {\emph{irreducible constructible}} subset.
 
A subgroup $\GGG \subseteq \Aut(X)$ is \emph{connected}
if for every $g \in \GGG$ there is an irreducible variety $A$ and a morphism 
$\rho \colon A \to \Aut(X)$ with image in $\GGG$ such that $\rho(A)$ contains both $g$
and the identity $\id_X$. 
The {\emph{connected component}} of a subgroup $\GGG \subseteq \Aut(X)$ is the subgroup $\GGG^\circ\subset\GGG$ generated, as an abstract group, by the images of all  irreducible subsets $V\subset \GGG$ containing $\id_X$.
 
\begin{rem}
 If $\rho\colon A\to \Aut(X)$ and $\varphi\colon B\to \Aut(X)$ are morphisms, then $\psi\colon A\times B\to \Aut(X)$, $(a,b)\mapsto \rho(a)\circ \phi(b)$ is a morphism, too. If $A$ and $B$ are irreducible, so is $A\times B$, and if  $\rho(A)$ and $\varphi(B)$ contain $\id_X$, then so does $\psi(A\times B)$.
 Thus, $\GGG$ is connected if and only if it coincides with $\GGG^\circ$.
\end{rem}

The \emph{dimension $\dim \GGG$} of a subgroup 
$\GGG \subseteq \Aut(X)$ is, as defined by Ramanujam, the supremum of $\dim(A)$ over all 
\emph{injective} morphisms $A \to \Aut(X)$ with image in $\GGG$, where $A$ is any irreducible variety.

\subsection{Ramanujam's theorem} 
In~\cite{Ra1964A-note-on-automorp}, Ramanujam proved the following theorem when the variety $X$ is  irreducible and the group $G$ is connected. 
 
\begin{ramthm}
\label{Thm.Ramanujam}
Let $X$ be a variety over an algebraically closed field~$\kk$. 
Let $G$ be a finite dimensional subgroup of $\Aut(X)$ such that $G^\circ$ has finite index in $G$.
Then, $G$ carries the structure of an 
algebraic group which is uniquely determined by the following universal property:
\be
\item[{\rm{(UP)}}] 
{The action $G\times X \to X$ is an action of the algebraic group $G$, and if $\mu\colon A \to \Aut(X)$ is a morphism such that $\mu(A) \subset G$, then the induced map 
$\mu\colon A \to G$ is a morphism of algebraic varieties.}
\ee
\end{ramthm}

The goal of this note is to explain why Ramanujam's original statement (for an irreducible $X$) implies the theorem above. To do this, we provide alternative characterizations of the dimension of a subgroup $\GGG\subset \Aut(X)$. 
The main reason why Ramanujam assumes  all his varieties to be irreducible, is because he relies on Galois theory for extensions of the field of functions of $X$ (resp. of $A$, where $A\to \Aut(X)$ is a morphism). So, we already departed from Ramanujam's viewpoint by not assuming $A$ to be irreducible in the above definitions.

\begin{rem} Assume $G \subseteq \Aut(X)$ is  a finite dimensional subgroup such that 
$G^\circ$ has finite index in $G$. 
Endow $G$ with the algebraic group structure provided by Ramanujam’s theorem. 
Then $G$ acts schematically faithfully on $X$. 
This implies that every trivially acting normal subgroup scheme 
$N \subseteq G$ is trivial. 
In fact, by Ramanujam's theorem, the identity factors as $G \to G/N \to G$.
\end{rem}

In order to illustrate this remark, we assume  $\Char(\kk) = p > 0$ and consider the following
two $\Ga$-actions on the line $\mathbb{A}^1$:
\[
(t, x) \mapsto x + t \quad \textrm{and} \quad (t, x) \mapsto x + t^p \, .
\]
Both actions have the same image in $\Aut(\bbA^1)$, namely the group $\mathsf T$ of translations, but only the former 
defines the algebraic group action
given by Ramanujam's theorem. The restriction of the latter to the non-reduced subgroup scheme $\Spec(\kk[t] / (t^p))$ is trivial, hence
$\Ga$ does not act schematically faithfully in this case.

\begin{cor}
Let $X$ be a variety. If $G  \subseteq \Aut(X)$ is a subgroup that is also a constructible subset, then $G$ has the structure of an algebraic group given by Ramanujam's theorem. If, moreover, $G$ is irreducible as a constructible subset, then it is connected as an algebraic group.
\end{cor}

\begin{proof} If $G$ is the image of a morphism $\rho\colon A \to \Aut(X)$, then, by Lemma~2.7 of
\cite{KrReSa2021Is-the-Affine-Spac}, $\dim G \leq \dim A$, and the claim follows from Ramanujam's theorem. 
\end{proof}

\section{Preliminary remarks on morphisms and orbits}
\subsection{Orbits} The following lemma 
is proven for irreducible varieties in Lem\-ma~1 of~\cite{Ra1964A-note-on-automorp}.
The same argument applies verbatim to arbitrary varieties.
\begin{lem}
\label{Lem.Lemma1_Ramanujam}
Let $X$ be a variety and let $\rho \colon A \to \Aut(X)$ be a morphism.
Then there exist an integer $k\geq 1$ and a point $(x_1, \ldots, x_k)\in X^k$ such that 
$\rho(a)=\rho(a')$ if and only if $\rho(a)(x_i)=\rho(a')(x_i)$ for all $i$.
\hfil \qed
\end{lem}

\begin{rem}
   \label{Rem.Same_fibres}
   Lemma~\ref{Lem.Lemma1_Ramanujam} has the following immediate consequence:
   \textit{For any morphism $\rho\colon A \to \Aut(X)$ there is a morphism of varieties $\mu\colon A \to B$ with the same (reduced) fibers.}
\end{rem}

If an algebraic group $G$ acts on an algebraic variety, then we get a natural homomorphism
from the Lie algebra $\Lie(G)$ to the space of vector fields $\Ve(X)$ given by
\[
   \xi_G \colon \Lie(G) \to \Ve(X) \, , \quad v \mapsto \left(x \mapsto (\textrm{d}_e \rho_x )v \right)
\]
where $\rho_x \colon G \to X$ denotes the orbit map associated to $x$ and
$e \in G$ is the neutral element (we interpret the vector fields as sections of the tangent bundle of $X$).

\begin{lem}
   \label{Lem.Parametrization_G}
   Let $G$ be an algebraic group acting faithfully on $X$ such that
   $\xi_G \colon \Lie(G) \to \Ve(X)$ is injective.
   Then there
   exist an integer $k \geq 1$ and $z = (x_1, \ldots, x_k) \in X^k$ such that
   the orbit map $\rho_z \colon G \to Gz$ is an isomorphism.
\end{lem}

\begin{proof}
   By Lemma~\ref{Lem.Lemma1_Ramanujam} we can choose $k \geq 1$ and 
   $z = (x_1, \ldots, x_k) \in X^k$ such that $\rho_z$ is injective.
   As $\xi_G$ is injective, we may extend the tuple $z$ 
   by adding finitely many further points so that
   $\Lie(G) \to T_z X^{k}$, 
   $v \mapsto ((\textrm{d}_e \rho_{x_1}) v, \ldots, (\textrm{d}_e \rho_{x_k}) v)$ is injective. 
   Hence, $\rho_z \colon G \to Gz$ is smooth and thus it is an isomorphism.
\end{proof}

Let $\GGG\subset \Aut(X)$ be a connected subgroup. 
Proposition~\ref{Orbits.prop}(1) below implies that each orbit
$\GGG x$ is locally closed in $X$, 
so $\dim \GGG x$ makes sense. 
We denote by $\mdo(\GGG;X)$ the  {\emph{maximal dimension of orbits}}:
\begin{equation*}
\mdo(\GGG;X)\coloneqq \max\{\dim \GGG x\mid x \in X\}.
\end{equation*}
Then,  $\Xm\subset X$ will denote the union  of orbits of dimension $\mdo(\GGG ; X)$.

A set $\SSS$ of subgroups of $\GGG$ is \emph{covering} if $\GGG = \cup_{G \in \SSS} G$. It
is \emph{directed} if, for any pair $G_1,G_2 \in \SSS$, there is a $G \in \SSS$ containing both.

\begin{prop}\label{Orbits.prop}
Let $X$ be a variety and $\GGG$ be a connected subgroup of~$\Aut(X)$. 
\begin{enumerate}[\rm (1)]
\item
The $\GGG$-orbits in $X$ are open in their closure. 
They are locally closed in $X$ and smooth.

\item \label{algebraic-subset.item}
The group $\GGG$ contains an irreducible constructible subset $V \subseteq \Aut(X)$
with $\id_X \in V$
such that $\GGG x = Vx$ for all $x \in X$.
Moreover, if $\rho \colon A \to \Aut(X)$ is a morphism such that $A$ is irreducible,
$\rho(A)$ contains $\id$, and $\rho(A)$ 
generates $\GGG$, then we can choose
$V = (\rho(A) \cdot \rho(A)^{-1})^m$ for some $m \geq 1$.
\item
The union $\Xm$ of orbits of maximal dimension is open in $X$.
\item If $\SSS$ is a covering and directed set of connected algebraic subgroups for $\GGG$, then there is $G \in \SSS$ such that $\GGG x = G x$ for all $x \in X$. 
\end{enumerate}
\end{prop}

\begin{proof}
(a) For irreducible varieties, Assertion (1) is part of Lemma~2 of \cite{Ra1964A-note-on-automorp}.
For completeness we insert a proof in the general case. Let $x \in X$.
For $i=1,2$, let $\rho_i \colon A_i \to \Aut(X)$ be morphisms with $\rho_i(A_i) \subset \GGG$,  $\id_X \in \rho_i(A_i)$, and
 $A_i$  irreducible. The map
\[
A_1 \times A_2 \to \Aut(X), \quad (a_1,a_2) \mapsto \rho_1(a_1)\circ\rho_2(a_2)
\]
is again a morphism into $\GGG$ whose image contains both $\rho_1(A_1)$ and $\rho_2(A_2)$. 
By Chevalley's constructibility theorem the sets $\rho(A_i)x$ are constructible. Thus, iterating this construction and using the connectedness of $\GGG$, 
we obtain a morphism $\rho \colon A \to \Aut(X)$ with
$A$ irreducible, $\rho(A)\subset \GGG$, and $\overline{\rho(A)x} = \overline{\GGG x}$. The set $\rho(A)x\subset X$ is constructible and  contained in $\GGG x$, so it contains a dense open subset $U$ of $\overline{\GGG x}$. 
Thus $\GGG x = \GGG U$ is open and dense in~$\overline{\GGG x}$.
\ps
(b) Assertion (2) is given in Theorem~1 of \cite{Po2014On-infinite-dimens}. (The statement assumes  $X$ to be irreducible, but the definitions, given in \S 2 of   \cite{Po2014On-infinite-dimens}, do not use this fact and the general case follows from the irreducible case.)
\ps
(c) Now, we prove Assertion~(3) (see also Theorem~2 of~\cite{Po2014On-infinite-dimens}). Let $\rho \colon A \to \Aut(X)$ be a morphism
where $A$ is an irreducible variety and
$V \coloneqq \rho(A) \subseteq \Aut(X)$ satisfies Assertion (2).
Set $m = \mdo(\GGG; X)$. For $x \in X$, set $ax:=\rho(a)x$,
and consider the surjection 
$\eta_x \colon A \to A x$, $a \mapsto ax$. The local dimension of the fiber of $\eta_x$ satisfies
\begin{equation}\label{eq:locest}
\tag{$\ast$}
\dim_a (\eta_x^{-1}(ax)) \geq \dim A - \dim A x  \geq \dim A - m 
\end{equation}
and for a general $a \in A$, the first inequality is in fact an equality. Now, consider
the map $ \varphi \colon A \times X \to X \times X$, $(a, x) \mapsto (ax, x)$.
Then, $\varphi^{-1}(\varphi(a, x)) = \eta_x^{-1}(ax) \times \{ x \}$ for all 
$(a, x) \in A \times X$.
By Chevalley's semi-continuity theorem (see Thé\-o\-rème~13.1.3 in \cite{Gr1966Elements-de-geomet-28}), applied to $\varphi$,  the subset 
$U \subseteq A \times X$ of points $(a, x)$ with $\dim_a \eta_x^{-1}(ax) = \dim A - m$
is open in $A \times X$. From the estimate~\eqref{eq:locest}, 
we see that $(a,x)\in U$ 
 if and only if $x\in \Xm$ and $\dim_a \eta_x^{-1}(ax)= \dim A - \dim Ax$; thus 
\begin{equation*}
U= \{(a, x) \in A \times \Xm \mid \dim_a \eta_x^{-1}(ax) = \dim A - \dim Ax \}.
\end{equation*} 
For any $x \in X$, we have
$\dim_a \eta_x^{-1}(ax) = \dim A - \dim Ax$ for a general $a \in A$. Therefore,
$\Xm$ is the projection of $U$ to $X$, and the statement follows.
\ps
(d) Finally, we prove Assertion~(4). 
  Since $\SSS$ is directed and since the subgroups $G \in  \SSS$ are connected,
  we may argue componentwise and thus assume that $X$ is irreducible. By induction on $\dim X$, we may further assume that $X=\Xm$. 
  Again, we set $m = \mdo(\GGG;X)$.

  We claim that there exists $G \in \SSS$ with
  $\dim Gx = m$ for all $x \in X$. 
  Indeed, otherwise we can fix  $G_0 \in \SSS$
  such that
  \[
      Z_{G_0} \coloneqq \overline{\set{x \in X}{ \dim G_0 x < \dim \GGG x}} \subset X
  \] 
  is non-empty and of minimal dimension. 
  Then, arguing as in (a) and using that $\SSS$ is covering and directed, there exists $H \in \SSS$ containing $G_0$, such that
  each irreducible component of $Z_{G_0}$ contains a point $x$ with $\dim H x = m$.
 Since
  the $H$-orbits of maximal dimension $m$ form an open subset of $X$, we get $\dim Z_H < \dim Z_{G_0}$ or $Z_H=\emptyset$. 
  Both cases  contradict the choice of $G_0$, which proves the claim.

  Hence, $G x$ is dense in $\GGG x$ for all $x \in X$, which implies the statement.
\end{proof}

\subsection{Reparametrisation and dimension} 

\begin{prop}
\label{Pro.Injective_morphisms}
Given any morphism $\rho \colon A \to \Aut(X)$ from an irreducible variety $A$, 
there is an irreducible variety $B$ and an injective morphism $\mu \colon B \to \Aut(X)$ such that $\mu(B) \subseteq \rho(A)$ 
and $\rho^{-1}(\mu(B))$ is open and dense in $A$.
\end{prop}

\begin{proof}
We may (and we do) assume from the outset that $A$ is affine.

Lemma~\ref{Lem.Lemma1_Ramanujam} and Remark~\ref{Rem.Same_fibres} provide a dominant morphism  
$\phi \colon A \to Z\subset X^k$ 
such that the fibres of $\varphi$ coincide with the fibres of $\rho$.
Replacing $A$ by a suitable closed and irreducible
subset we can assume that $\varphi$ has finite degree. Replacing $A$ by a suitable open dense subset we can further assume that $U \coloneqq \varphi(A)\subset Z$ is  open and dense in $Z$ and that $\varphi \colon A \to U$ is a finite morphism. 

 Let $\kk(A)$ and $\kk(U)$ be the fields of rational functions on $A$ and $U$.
Let $\KK/\kk(U)$ be a normal extension which contains $\kk(A)$, and let $\tilde{A}$ be the normalization of $A$ in $\KK$. Then $\eta \colon \tilde{A} \to A$ is a finite morphism.

If $\Gamma$ denotes the Galois group of $\KK/\kk(U)$,  then $\Gamma$ acts on $\tilde{A}$, and we get a factorization of $\varphi \circ \eta \colon \tilde{A} \to U$ as $\tilde{A} \to \tilde{A}/\Gamma \to U$.
Since $\kk(U) \subseteq \kk(\tilde{A}/\Gamma)$
is purely inseparable (see Proposition~6.11 in Chapter V of \cite{La2002Algebra}),  
after possibly shrinking $U$ we may in addition assume that $\tilde{A}/\Gamma \to U$ is bijective.   
The $\tilde{A}$-automorphism of $\tilde{A} \times X$
induced by $\rho \circ \eta$ descends to an $\tilde{A}/\Gamma$-automorphism of $(\tilde{A}/\Gamma) \times X$
and hence yields a morphism $\mu \colon \tilde{A} / \Gamma \to \Aut(X)$. By construction,
the fibres of $\mu$ are exactly the fibres of $\tilde{A}/\Gamma \to U$, and hence $\mu$ is injective.
Thus our claim follows by setting $B:=\tilde{A}/\Gamma$.
\end{proof}

Applying Proposition~\ref{Pro.Injective_morphisms} inductively  yields the following corollary.
\begin{cor}
\label{Cor.Existence_injective_morphism}
Let $\rho \colon A \to \Aut(X)$ be a morphism. Then there exists a variety $B$ and an injective morphism $\mu \colon B \to \Aut(X)$ with the same image as $\rho$. 
\hfill \qed
\end{cor}

In this corollary the variety $B$ is in general not irreducible, even when $A$ is assumed to be irreducible. 

So, if a  subset  $S$ in $\Aut(X)$ is finite dimensional, of dimension $d$, then for every morphism
$A\to \Aut(X)$ with image in $S$, one can find an injective  morphism $B\to \Aut(X)$ with the same image and $\dim(B)\leq d$. {\emph{This gives an
alternative definition of the dimension}}, as the minimal such $d$.

\section{Alternative characterization of dimension for groups}

\begin{prop}\label{Prop.Alternative_definition_dimension}
Let $X$ be a variety, and let $\GGG \subseteq \Aut(X)$ be a connected subgroup. 
Then $\dim\GGG=\sup\{\dim \GGG y \mid y \in X^n, n \geq 1\}$.
\end{prop}
Here, $\GGG y$ is the orbit obtained from the diagonal action of $\GGG$ on $X^n$. 
By Proposition~\ref{Orbits.prop}, $\GGG y$ is an irreducible and locally closed subset of~$X^n$.
  
\begin{proof}  We follow the argument of Lemma~2 in \cite{Ra1964A-note-on-automorp}.
      Suppose $A$ is irreducible
      and $A \to \Aut(X)$ is an injective morphism with image in $\GGG$.
   Take $k$ and $x:=(x_1, \ldots, x_k)$ as in Lemma~\ref{Lem.Lemma1_Ramanujam}. 
   Then, the orbit map 
   $A \to X^k$ with respect to $x$ is injective and
   thus $\dim(A) = \dim (A x) \leq \dim (\GGG x)$. 
  To conclude the proof, we need to show that $\dim(\GGG y)\leq \dim(\GGG)$ for any $n\geq 1$ and $y\in X^n$. Fix such a point $y \in X^n$. 
      By Proposition~\ref{Orbits.prop}\,(2) there is a morphism 
      $\rho\colon A\to\operatorname{Aut}(X)$ with image in $\GGG$ such that $A$ is irreducible and $\rho(A) y= \GGG y$. 
     By Corollary~\ref{Cor.Existence_injective_morphism}
      we find an injective morphism $\mu\colon B\to\operatorname{Aut}(X)$ 
      with $\mu(B) = \rho(A)$. 
      It follows that $\dim \GGG y \leq \dim B \leq \dim \GGG$.
\end{proof}

Let $X' \subseteq X$ be a locally closed subset. For instance $X'$ can be an irreducible component of $X$.
If $\rho \colon A \to \Aut(X)$ is a morphism such that 
$X'$ is invariant under $\rho(a)$ for all $a \in A$,
then we get an induced morphism $A \to \Aut(X')$. In particular, 
if $X'$ is invariant under a connected subgroup $\GGG \subseteq \Aut(X)$, then 
the image $\GGG'$ in $\Aut(X')$ is again connected. Furthermore, 
using Proposition~\ref{Prop.Alternative_definition_dimension} we can bound its dimension.

\begin{cor}
   \label{Cor.Dimension_estimate}
   Let $X$ be a variety and $\GGG \subseteq \Aut(X)$ a connected subgroup.
   If $X' \subseteq X$ is a locally closed $\GGG$-invariant subvariety, 
   then $\dim \GGG' \leq \dim \GGG$,
   where $\GGG'$ denotes the image of $\GGG$ in $\Aut(X')$. \hfill \qed
\end{cor}

\section{The proof of Ramanujam's result for arbitrary varieties}

To prove Ramanujam's theorem for $\GGG\subset \Aut(X)$ when $X$ is reducible, we will apply the initial  theorem of Ramanujam to the restriction $\GGGr{X'}$ on each irreducible component $X'$ of $X$. This will endow $\GGGr{X'}$ with a structure of algebraic group acting algebraically on $X'$. Then we will glue these algebraic actions to get an algebraic action of $\GGG$ on $X$. 
However, one has to be careful with gluings,
as the following examples illustrate. In these examples, $(x,y)$ and  $(x,y,z)$ are affine coordinates on the plane $\bbA^2$  and the  space  $\bbA^3$, respectively.

\begin{exa}
Let $X \subseteq \bbA^2$ be defined by
$x y(x-y) =0$. 
The lines $X_1 = \{y=0\}$, $X_2 = \{x=0\}$ and 
$X_3 = \{x=y\}$ are its irreducible components; we identify each of them 
with $\bbA^1$, with origin at $(0,0)$. 
Define $g_1$ to be the identity map on $X_1$, 
$g_2$ to be the identity on $X_2$, and $g_3$ to be the multiplication by $\alpha$ on $X_3$, for some $\alpha\in \kk\setminus\{0, 1\}$.
Then $g_1$, $g_2$, and $g_3$ pairwise agree  on the intersections of their domains
but do not glue to an automorphism $g$ of $X$, since the differential of such a hypothetical $g$ at $(0,0)$ would be a linear map fixing $(0,1)$ and $(1,0)$ but not $(1,1)$. 
So, already at the level of tangent spaces 
the gluing can fail. 
\end{exa}
 
\begin{exa}
Fix an integer $m > 1$, and let $X\subseteq \bbA^3$ be the union of the plane 
$X_1 = \{y=0\}$ and  the cylinder $X_2 = \{y = x^{m+1}\}$.
Fix an element $a \neq 0, 1$ in~$\kk$. We consider the following
faithful $\Ga$-actions. On $X_1$ the action is given by $(t, x, z) \mapsto (x, z + tx^m)$,
and on $X_2$ it is given by $(t, x, z) \mapsto (x, z + t a x^m)$
using the isomorphisms $X_i \to \bbA^2$, $(x, y, z) \mapsto (x, z)$.
These actions agree on the level of tangent spaces (since $m > 1$), but they don't glue,
since on the schematic intersection $X_1 \cap X_2 \simeq \Spec(\kk[x]/(x^{m+1})[z])$ the
actions do not agree.
\end{exa}

\begin{prop}\label{prop1} Let $X$ be a variety, with $X = X_1 \cup X_2$ for some closed subsets $X_1,X_2 \subset X$.
Let $G$ be a connected algebraic group acting regularly on $X_1$ and $X_2$. This action defines a regular action on $X$ if and only if the two actions  induced  on the schematic intersection $X_1\cap X_2$ coincide. The latter holds if every $g \in G$ defines a regular automorphism of $X$.
\end{prop}

This will be the tool that makes the gluing process work.

\begin{lem}\label{lem5}
Let $X,Z$ be varieties, where $X = X_1 \cup X_2$ with two closed subsets $X_1,X_2 \subset X$. Assume that $\phi_i\colon X_i \to Z$, $i=1,2$, are morphisms with the property $\phi_1|_{X_1\cap X_2} = \phi_2|_{X_1\cap X_2}$ where $X_1\cap X_2$ denotes the schematic intersection. Then there is a unique morphism $\phi\colon X \to Z$ such that $\phi|_{X_i} = \phi_i$.
\end{lem}
\begin{proof}
We can reduce to the case where $X$ and $Z$ are affine. Then $\OOO(X_i) = \OOO(X)/\aa_i$ with radical ideals $\aa_i$, and $\aa_1\cap\aa_2 = (0)$. The schematic intersection $X_1\cap X_2$ is given by  
$\Spec\OOO(X)/(\aa_1+\aa_2)$. Now the exact sequence
$$
0 \to \OOO(X) \to (\OOO(X)/\aa_1) \times (\OOO(X)/\aa_2) 
\xlongrightarrow{(f_1, f_2) \mapsto f_1 - f_2} \OOO(X)/(\aa_1+\aa_2)
$$ implies the claim.
\end{proof}
\begin{lem}\label{lem6}
Let $A$ be a variety and $Z$ a $\kk$-scheme. Assume that $Y \subset A \times Z$ is a closed subscheme which contains $\{a\}\times Z$ for any $a \in A$. Then $Y = A\times Z$.
\end{lem}
\begin{proof}
We can assume that $A$ and $Z$ are affine. Then $Y$ corresponds to an ideal $\aa \subset \OOO(A)\otimes\OOO(Z)$. By assumption, $\aa\in \mm \otimes \OOO(Z)$ for any maximal ideal $\mm \subset \OOO(A)$. Since $\bigcap_{\mm} (\mm \otimes \OOO(Z)) = \left(\bigcap_{\mm}\mm\right) \otimes \OOO(Z) = (0)$ the claim follows. 
\end{proof}

\begin{proof}[Proof of Proposition~\ref{prop1}]
Looking at the group action $G \times X \to X$, the first claim follows from  Lemma~\ref{lem5}. For the second we look at the two morphisms $\phi_1,\phi_2\colon G \times (X_1\cap X_2) \to X_1\cap X_2$ of schemes of finite type given by the actions of $G$. Then  the inverse image of the diagonal $\Delta \subset (X_1\cap X_2)\times (X_1\cap X_2)$ under the morphism $(\phi_1,\phi_2)\colon G \times (X_1\cap X_2) \to (X_1\cap X_2)\times (X_1\cap X_2)$ is the closed subscheme where $\phi_1$ and $\phi_2$ are equal. By assumption, $\phi_1$ and $\phi_2$ coincide on $\{g\}\times (X_1\cap X_2)$ for any $g \in G$; thus,  the claim follows from Lemma~\ref{lem6}. 
\end{proof}

\begin{proof}[Proof of Ramanujam's theorem for arbitrary varieties]
First, we assume that $G$ is connected. 
\ps
(a)
Let $G \subseteq \Aut(X)$ be a connected finite dimensional subgroup. Denote by $X_1, \ldots, X_n$ the irreducible components of $X$. Since $G$ is connected, every $X_i$ is $G$-invariant, and we have  a group homomorphism $p_i\colon G \to \Aut(X_i)$ for every~$i$. The image $G_i \subseteq \Aut(X_i)$ is a connected subgroup and its dimension is finite by Corollary~\ref{Cor.Dimension_estimate}. By Ramanujam's theorem for irreducible varieties, $G_i$ has a unique structure of an algebraic subgroup that satisfies the universal property~(UP)
for morphisms $A \to \Aut(X_i)$ with image in $G_i$ and irreducible~$A$.
\ps
(b)
Define $p:=(p_1,p_2,\ldots,p_n)\colon G \to G_1 \times G_2 \times\cdots\times G_n$, and 
denote by $G':=p(G) \subseteq G_1 \times \cdots \times G_n$ the image of $G$.
For every morphism $A \to \Aut(X)$ with image in $G$ and irreducible $A$,
the composition  $A \to G \to G_i$ is a morphism by the universal property (UP) of $G_i$.  As $G$ is connected, $G'$ is generated by the images of all morphisms
\[
A \xrightarrow{\rho} G \xrightarrow{p} G_1 \times \cdots \times G_n,
\]
where $A$ is irreducible,
$\rho \colon A \to \Aut(X)$ is a morphism, $\rho(A)\subset G$, and $\id_X \in \rho(A)$.
Hence, $G'$ is a connected closed algebraic subgroup of 
$G_1 \times \cdots \times G_n$ by Proposition 7.5 in \cite{Hu1975Linear-algebraic-g}.
\ps
(c) Every element of $G'$ acts via a regular automorphism on $X$. Hence, by Proposition~\ref{prop1},
the actions $G' \times X_i \to G_i \times X_i \to X_i$ 
glue inductively to a faithful action $G' \times X \to X$,  and the image  of $G'$ 
in $\Aut(X)$ is $G$.
\ps
(d) 
We give $G$ the structure of an algebraic group via $G'$. 
By the theorem 
in~\cite{Ra1964A-note-on-automorp}, the natural homomorphisms 
$\Lie(G_i) \to \Ve(X_i)$ are injective. By construction 
$p \colon G \to G_1 \times \ldots \times G_n$ is a closed embedding, 
and thus the natural homomorphism $\Lie(G) \to \Ve(X)$ is injective as well.
Lemma~\ref{Lem.Parametrization_G} gives $k \geq 1$ and $z \in X^k$ such that 
the orbit map $G \to Gz$ is an isomorphism.
If $\mu\colon B \to \Aut(X)$ is a morphism with image in $G$, then $B \to Gz$ given by 
$b \mapsto \mu(b)(z)$ is a morphism and hence $B \to G$ is a morphism as well. This proves (UP).
\ps
This concludes the proof when $G$ is connected. When $G^\circ$ has finite index in $G$, 
we can directly apply the argumentation for Theorem~9 in \cite{KrReSa2021Is-the-Affine-Spac} or 
Lemma~6.2 in~\cite{ReUrSa2025The-structure-of-a}.
Here is the argument. Choose elements $g_1=e$, $g_2$, $\ldots$, $g_m \in G$ such that $G$ is the disjoint union $G = \bigcup_i g_i G^\circ$. Left translation by $g_i$ provides a bijective map $\iota_i\colon G^\circ \simto g_iG^\circ$; using $\iota_i$, we transport the algebraic structure of $G^\circ$ on each $g_iG^\circ$; this defines a structure of algebraic variety on $G$ which does not depend on the choice of the $g_i$. Then, one checks that the multiplication $G \times G \to G$ and the inverse map $G \to G$ are morphisms. Doing so, $G$ inherits the structure of an algebraic group. It remains to see that the universal property (UP) holds. It is clear from the construction above that the action map $G \times X \to X$ is a morphism. Now assume that $\rho\colon A \to \Aut(X)$ is a morphism with image in $G$, where we can assume that $A$ is connected. Then $\rho(A) \subset g_iG^\circ$ for some $i$, and $\rho\colon A \to g_i G$ is a morphism of varieties, because the composition $(\iota_i)^{-1}\circ\rho\colon A \to G^\circ$ is a morphism of varieties.
This finishes the proof.
\end{proof}

\section{Appendix}
{\small{
Our goal  is to provide a self contained proof of the following lemma, which also follows from Proposition 5.7 of  \cite[Exposé I]{GrRa2003Revetements-etales} and Proposition~9.6.1 of \cite{Gr1966Elements-de-geomet-28}.
\begin{lem}\label{Lem.inverse-family}
Let $\rho \colon A \to \Aut(X)$ be a morphism. Then 
$
\Phi\colon A\times X \to A \times X, \quad (a,x) \mapsto (a,\rho(a)x),
$
is an isomorphism. In particular, the inverse family $\rho^{-1}\colon A \to \Aut(X)$ given by $a \mapsto \rho(a)^{-1}$, is a morphism.
\end{lem}
\begin{proof}
The morphism $\Phi$ has the following properties:
\begin{enumerate}[\rm (i)]
\item
For every $a\in A$ the morphism $\Phi_a\colon\{a\}\times X \to \{a\}\times X$ is an isomorphism.
\item
$\Phi$ is bijective. This follows from (i).
\item
For every $(a,x) \in A \times X$ the differential 
$
d\Phi_{(a, x)} \colon T_{(a,x)} (A\times X) \to T_{(a,\rho(a)x)} (A\times X)
$
is an isomorphism. This follows from (i), because $T_{(a,x)} (A\times X)  = T_aA \oplus T_xX$ and $d\Phi_{(a, x)}$ is ``lower triangular'' with respect to this direct sum: it acts as the identity on $T_aA$ and the composition of $d\Phi_a$ with the projection onto $T_xX$ is the invertible map 
$d(\rho(a))_x$. 
\item
$\Phi$ has reduced fibers. This follows from (i), because $\Phi^{-1}(\{a\}\times X)$ is reduced and equal to $\{a\}\times X$.
\end{enumerate}
Now we proceed as follows.
\ps
(a) If $A$ and $X$ are smooth, then $\Phi$ is smooth by (iii), hence an isomorphism by (ii).
\ps
(b) If $A$ and $X$ are normal, then the claim follows from Zariski's main theorem, because $\Phi$ is birational by (a).
\ps
(c)
Let $\eta_A\colon \tilde A \to A$, $\eta_X\colon \tilde X \to X$ be the normalizations. We get the following commutative diagram
\[
\begin{CD}
A \times X @<{\simeq}<< \Gamma_\Phi @>{\phi}>> A \times X \\
@AA{\eta_A\times\eta_X}A @AA{\eta}A @AA{\eta_A\times\eta_X}A\\
\tilde A \times \tilde X @<{\simeq}<< \Gamma_{\tilde\Phi}@>{\simeq}>> \tilde A \times \tilde X
\end{CD}
\]
where $\Gamma_\Phi:=\set{(a,x,y)}{y = \rho(a)x}$ is the graph of the $A$-morphism $\Phi$ and $\Gamma_{\tilde\Phi}$ the graph of the normalization $\tilde\Phi\colon \tilde A \times \tilde X \to \tilde A \times \tilde X$ which is an isomorphism by (b).
By construction, the vertical maps are finite and surjective hence $\phi$ is finite as well. Since $\phi$ is birational and has reduced fibers the claim follows 
(see  Lemma 3.3.3 in \cite{FuKr2018On-the-geometry-of}). 
\end{proof}
}}
 
\bibliographystyle{plain}
\bibliography{biblio-solvable}
\end{document}